\documentclass[12pt]{amsart}
\usepackage{amscd,amssymb,amsmath,multicol,float,scalefnt}
\restylefloat{table}
\newtheorem{thm}[equation]{Theorem}
\numberwithin{equation}{section}

\begin{document}
\raggedbottom \voffset=-.7truein \hoffset=0truein \vsize=8truein
\hsize=6truein \textheight=8truein \textwidth=6truein
\baselineskip=18truept

\def\mapright#1{\ \smash{\mathop{\longrightarrow}\limits^{#1}}\ }
\def\mapleft#1{\smash{\mathop{\longleftarrow}\limits^{#1}}}
\def\mapup#1{\Big\uparrow\rlap{$\vcenter {\hbox {$#1$}}$}}
\def\mapdown#1{\Big\downarrow\rlap{$\vcenter {\hbox {$\ssize{#1}$}}$}}
\def\mapne#1{\nearrow\rlap{$\vcenter {\hbox {$#1$}}$}}
\def\mapse#1{\searrow\rlap{$\vcenter {\hbox {$\ssize{#1}$}}$}}
\def\mapr#1{\smash{\mathop{\rightarrow}\limits^{#1}}}
\def\ss{\smallskip}
\def\vp{v_1^{-1}\pi}
\def\at{{\widetilde\alpha}}
\def\sm{\wedge}
\def\la{\langle}
\def\ra{\rangle}
\def\on{\operatorname}
\def\spin{\on{Spin}}
\def\kbar{{\overline k}}
\def\qed{\quad\rule{8pt}{8pt}\bigskip}
\def\ssize{\scriptstyle}
\def\a{\alpha}
\def\bz{{\Bbb Z}}
\def\im{\on{im}}
\def\ct{\widetilde{C}}
\def\ext{\on{Ext}}
\def\sq{\on{Sq}}
\def\eps{\epsilon}
\def\ar#1{\stackrel {#1}{\rightarrow}}
\def\br{{\bold R}}
\def\bC{{\bold C}}
\def\bA{{\bold A}}
\def\bB{{\bold B}}
\def\bD{{\bold D}}
\def\bh{{\bold H}}
\def\bQ{{\bold Q}}
\def\bP{{\bold P}}
\def\bx{{\bold x}}
\def\bo{{\bold{bo}}}
\def\si{\sigma}
\def\Ebar{{\overline E}}
\def\dbar{{\overline d}}
\def\Sum{\sum}
\def\tfrac{\textstyle\frac}
\def\tb{\textstyle\binom}
\def\Si{\Sigma}
\def\w{\wedge}
\def\equ{\begin{equation}}
\def\b{\beta}
\def\G{\Gamma}
\def\g{\gamma}
\def\k{\kappa}
\def\psit{\widetilde{\Psi}}
\def\tht{\widetilde{\Theta}}
\def\psiu{{\underline{\Psi}}}
\def\thu{{\underline{\Theta}}}
\def\aee{A_{\text{ee}}}
\def\aeo{A_{\text{eo}}}
\def\aoo{A_{\text{oo}}}
\def\aoe{A_{\text{oe}}}
\def\fbar{{\overline f}}
\def\endeq{\end{equation}}
\def\sn{S^{2n+1}}
\def\zp{\bold Z_p}
\def\A{{\cal A}}
\def\P{{\mathcal P}}
\def\cj{{\cal J}}
\def\zt{{\bold Z}_2}
\def\bs{{\bold s}}
\def\bof{{\bold f}}
\def\bq{{\bold Q}}
\def\be{{\bold e}}
\def\Hom{\on{Hom}}
\def\ker{\on{ker}}
\def\coker{\on{coker}}
\def\da{\downarrow}
\def\colim{\operatornamewithlimits{colim}}
\def\zphat{\bz_2^\wedge}
\def\io{\iota}
\def\Om{\Omega}
\def\Prod{\prod}
\def\e{{\cal E}}
\def\exp{\on{exp}}
\def\kbar{{\overline w}}
\def\xbar{{\overline x}}
\def\ybar{{\overline y}}
\def\zbar{{\overline z}}
\def\ebar{{\overline e}}
\def\nbar{{\overline n}}
\def\rbar{{\overline r}}
\def\et{{\widetilde E}}
\def\ni{\noindent}
\def\coef{\on{coef}}
\def\den{\on{den}}
\def\lcm{\on{l.c.m.}}
\def\vi{v_1^{-1}}
\def\ot{\otimes}
\def\psibar{{\overline\psi}}
\def\mhat{{\hat m}}
\def\exc{\on{exc}}
\def\ms{\medskip}
\def\ehat{{\hat e}}
\def\etao{{\eta_{\text{od}}}}
\def\etae{{\eta_{\text{ev}}}}
\def\dirlim{\operatornamewithlimits{dirlim}}
\def\gt{\widetilde{L}}
\def\lt{\widetilde{\lambda}}
\def\st{\widetilde{s}}
\def\ft{\widetilde{f}}
\def\sgd{\on{sgd}}
\def\lfl{\lfloor}
\def\rfl{\rfloor}
\def\ord{\on{ord}}
\def\gd{{\on{gd}}}
\def\rk{{{\on{rk}}_2}}
\def\nbar{{\overline{n}}}
\def\lg{{\on{lg}}}
\def\cR{\mathcal{R}}
\def\cT{\mathcal{T}}
\def\N{{\Bbb N}}
\def\Z{{\Bbb Z}}
\def\Q{{\Bbb Q}}
\def\R{{\Bbb R}}
\def\C{{\Bbb C}}
\def\l{\left}
\def\r{\right}
\def\mo{\on{mod}}
\def\vexp{v_1^{-1}\exp}
\def\notimm{\not\subseteq}
\def\Remark{\noindent{\it  Remark}}

\def\*#1{\mathbf{#1}}
\def\0{$\*0$}
\def\1{$\*1$}
\def\22{$(\*2,\*2)$}
\def\33{$(\*3,\*3)$}
\def\ss{\smallskip}
\def\ssum{\sum\limits}
\def\dsum{\displaystyle\sum}
\def\la{\langle}
\def\ra{\rangle}
\def\on{\operatorname}
\def\o{\on{o}}
\def\U{\on{U}}
\def\lg{\on{lg}}
\def\a{\alpha}
\def\bz{{\Bbb Z}}
\def\eps{\varepsilon}
\def\br{{\bold R}}
\def\bc{{\bold C}}
\def\bN{{\bold N}}
\def\nut{\widetilde{\nu}}
\def\tfrac{\textstyle\frac}
\def\b{\beta}
\def\G{\Gamma}
\def\g{\gamma}
\def\zt{{\Bbb Z}_2}
\def\zth{{\bold Z}_2^\wedge}
\def\bs{{\bold s}}
\def\bx{{\bold x}}
\def\bof{{\bold f}}
\def\bq{{\bold Q}}
\def\be{{\bold e}}
\def\lline{\rule{.6in}{.6pt}}
\def\xb{{\overline x}}
\def\xbar{{\overline x}}
\def\ybar{{\overline y}}
\def\zbar{{\overline z}}
\def\ebar{{\overline \be}}
\def\nbar{{\overline n}}
\def\rbar{{\overline r}}
\def\Mbar{{\overline M}}
\def\et{{\widetilde e}}
\def\ni{\noindent}
\def\ms{\medskip}
\def\ehat{{\hat e}}
\def\xhat{{\widehat x}}
\def\nbar{{\overline{n}}}
\def\minp{\min\nolimits'}
\def\N{{\Bbb N}}
\def\Z{{\Bbb Z}}
\def\Q{{\Bbb Q}}
\def\R{{\Bbb R}}
\def\C{{\Bbb C}}
\def\el{{\overline\ell}}
\def\TC{\on{TC}}
\def\dstyle{\displaystyle}
\def\ds{\dstyle}
\def\Remark{\noindent{\it  Remark}}
\title
{Topological complexity of spatial polygon spaces}
\author{Donald M. Davis}
\address{Department of Mathematics, Lehigh University\\Bethlehem, PA 18015, USA}
\email{dmd1@lehigh.edu}
\date{July 6, 2015}


\maketitle
\begin{abstract} Let $\el=(\ell_1,\ldots,\ell_n)$ be an $n$-tuple of positive real numbers, and let $N(\el)$ denote the space of equivalence classes of oriented  $n$-gons in $\R^3$ with consecutive sides of lengths $\ell_1,\ldots,\ell_n$, identified under translation and rotation of $\R^3$.
 Using known results about the integral cohomology ring, we prove that its topological complexity satisfies $\TC(N(\el))= 2n-5$, provided that $N(\el)$ is nonempty and contains no straight-line polygons.
 \end{abstract}

\section{Main result}\label{intro}
The topological complexity, $\TC(X)$, of a topological space $X$ is, roughly, the number of rules required to specify how to move between any two points of $X$. A ``rule'' must be such that the choice of path varies continuously with the choice of endpoints. (See \cite[\S4]{F}.) We study $\TC(X)$, where $X=N(\el)$ is the space of equivalence classes of oriented $n$-gons in  $\R^3$ with sides of length $(\ell_1,\ldots,\ell_n)$, identified under translation and rotation of $\R^3$. (See, e.g., \cite{Kl} or \cite{HK}.) Here $\el=(\ell_1,\ldots,\ell_n)$ is an $n$-tuple of positive real numbers. Thus
 $$N(\el)=\{(z_1,\ldots,z_n)\in (S^2)^n:\ell_1z_1+\cdots+\ell_nz_n=0\}/SO(3).$$
We can think of the sides of the polygon as linked arms of a robot, and then $\TC(X)$ is the number of rules required to program the robot to move from any configuration to any other.

We say that $\el$ is {\it generic} if  there is no subset $S\subset[n]$ such that $\ds\sum_{i\in S}\ell_i=\sum_{i\not\in S}\ell_i$, and {\it nondegenerate} if for all $i$, $\ell_i<\ds\sum_{j\ne i}\ell_j$.
Being generic says that there are no straight-line polygons in $N(\el)$, while being nondegenerate says that $N(\el)$ is nonempty.
If $\el$ is generic and nondegenerate, then  $N(\el)$ is a nonempty simply-connected $(2n-6)$-manifold (\cite{Kl} or \cite[10.3.33]{Hbook}) and hence satisfies \begin{equation}\label{bnd}\TC(\el)\le 2n-5\end{equation} by \cite[Thm 4.16]{F}.

In this paper, we prove the following theorem.

\begin{thm}\label{Nthm} If $\el$ is generic and nondegenerate, then $\TC(N(\el))\ge 2n-5$ $($and hence $\TC(N(\el))=2n-5$  by (\ref{bnd})$)$.\end{thm}

The proof of Theorem \ref{Nthm} uses the integral cohomology ring $H^*(N(\el))$, first described in \cite[Thm 6.4]{HK}.
Throughout the paper, all cohomology groups  have integral coefficients.
 To prove Theorem \ref{Nthm}, we will find  $2n-6$ classes $y_i\in H^2(N(\el))$ such that $\prod(y_i\ot1-1\ot y_i)\ne0$ in $H^{2n-6}(N(\el))\ot H^{2n-6}(N(\el))$.
This implies the theorem by the basic result that if in $H^*(X\times X)$ there is an $m$-fold nonzero product of classes of the form $y_i\otimes1-1\otimes y_i$,
 then $\TC(X)\ge m+1$.(\cite[Cor 4.40]{F})

\section{Proof}\label{Nsec}
In this section we prove Theorem \ref{Nthm}.
We begin by stating the information about the cohomology ring $H^*(N(\el))$ needed for our proof. Complete information is known, but the most delicate parts are not required here.
\begin{thm} \label{cohNthm} Let $\el$ be generic and nondegenerate.
\begin{enumerate}
\item The algebra $H^*(N(\el))$ is generated by  classes $R,V_1,\ldots,V_{n-1}$ in $H^2(N(\el))$.
\item The product of $n-2$ distinct $V_i$'s is 0.
\item If $d\le n-3$ and $S\subset\{1,\ldots,n-1\}$,
then all monomials $\dstyle{R^{e_0}\prod_{i\in S}V_i^{e_i}}$ with $e_i>0$ for $i\in S$ and $\dstyle\sum_{i\ge0} e_i=d$ are equal.
    \end{enumerate}
\end{thm}

\begin{proof} The complete result was given in \cite[Theorem 6.4]{HK}, which contains much more detailed information about which products of $V_i$'s must be 0. Our part (2) follows from their result and nondegeneracy, since the sum of any $n-1$ $\ell_i$'s will be greater than half the sum of all the $\ell_i$'s. Our $R$ is the negative of the $R$ in \cite{HK}.
This turns their relation $V_i^2+RV_i$ into our equality $V_i^2=RV_i$, which easily implies part (3).

\end{proof}
\begin{proof}[Proof of Theorem \ref{Nthm}.] Since $N(\el)$ is a nonempty oriented $(2n-6)$-manifold, $$H^{2n-6}(N(\el))\approx\Z.$$
Thus some monomial $R^{e_0}V_{i_1}^{e_1}\cdots V_{i_r}^{e_r}$ with $\ds\sum_{j=0}^re_j=n-3$  must be nonzero.
Choose one with minimal $r$. By Theorem \ref{cohNthm}(2), $r\le n-3$. Then
$$(R\ot1-1\ot R)^{2n-6-2r}\prod_{j=1}^r(V_{i_j}\ot1-1\ot V_{i_j})^2$$
is our desired nonzero product. This includes the possibility of $r=0$. Note that a nonzero term in the expansion of the product must
involve just the middle term of
$$(V_{i_j}\ot1-1\ot V_{i_j})^2=V_{i_j}^2\ot1-2V_{i_j}\ot V_{i_j}+1\ot V_{i_j}^2$$
by minimality of $r$. The only possible nonzero part of the product is that in $H^{2n-6}(N(\el))\ot H^{2n-6}(N(\el))$, and this will equal
$$(-1)^{n-3}\tbinom{2n-6-2r}{n-3-r}2^r R^{n-3-r}V_{i_1}\cdots V_{i_r}\ot R^{n-3-r}V_{i_1}\cdots V_{i_r}\ne0$$
by Theorem \ref{cohNthm}(3). Thus $\TC(N(\el))\ge 2n-5$ by \cite[Cor 4.40]{F}.
\end{proof}

In \cite{D}, we obtain some similar (but slightly weaker and much more restricted) results for planar polygon spaces, by similar, but much more complicated, calculations. The difficulty there is that all we know about is mod-2 cohomology, and it is much harder to find nonzero products.

 \def\line{\rule{.6in}{.6pt}}

\end{document}